\documentclass[reqno,12pt]{amsart}
\usepackage{enumitem,stix}
\usepackage{epsf}
\usepackage[pagebackref]{hyperref}
\usepackage{mathtools}
\usepackage{ytableau}
\usepackage{todonotes}
\usepackage{verbatim}
\usepackage{amsmath}
\usepackage{amsopn}
\usepackage{amssymb}
\usepackage{latexsym}
\usepackage{amsfonts}
\usepackage{amsthm}
\usepackage{tikz}
\usepackage[foot]{amsaddr}

\newtheorem{thm}{Theorem}[section]
\newtheorem{prop}[thm]{Proposition}
\newtheorem{lemma}[thm]{Lemma}

\newtheorem{con}[thm]{Conjecture}
\newtheorem{obs}[thm]{Observation}

\theoremstyle{definition}

\newtheorem{rem}[thm]{Remark}

\newtheorem{example}[thm]{Example}
\newtheorem{defn}[thm]{Definition}
\newtheorem{definition}[thm]{Definition}

\newcommand{\bbn}{\mathbb{N}}

\DeclareMathOperator\des{des}

\DeclareMathOperator\ides{ides}

\definecolor{fondo}{rgb}{0.898,0.996,0.898}
\definecolor{diagonal}{rgb}{0.466,0,0}
\usetikzlibrary{shapes.geometric}

\title{Blockwise Simple Permutations}

\begin{document}
\maketitle
\author{Eli Bagno, Estrella Eisenberg, Shulamit Reches and Moriah Sigron}

\begin{abstract}

A permutation is called {\it {block-wise simple}} if it contains no interval of the form $p_1\oplus p_2$ or $p_1 \ominus p_2$. We present this new set of permutations and explore some of its combinatorial properties. 
We present a generating function for this set, as well as a recursive formula for counting block-wise simple permutations. 
Following Tenner, who founded the notion of interval posets, we characterize and count the interval posets corresponding to block-wise simple permutations. We also present a bijection between these interval posets and certain tiling's of the $n$-gon. Finally, we prove that the bi-variate distribution of the descent and inverse descent
numbers are gamma-positive, provided the correctness of our recent conjecture on simple permutations.  
\end{abstract}

\section{Introduction}

This study focuses on a new set of permutations called {\it{block-wise simple permutations}}, which we define in two different but equivalent ways. The first definition is recursive and is stated here. The other, as well as the proof of the equivalence of the two definitions, is presented in Section \ref{counting}.
The notions and notations required are defined in Section \ref{background}.

\begin{definition}
\label{def of blockwise simple}
 Let $\mathcal{S}_n$ be the group of permutations of the set $\{1,\dots,n\}$. 
    
\begin{enumerate}
   \item The identity permutation $\pi=1$ is {\it block-wise simple}.
    
    \item A permutation $\pi \in\mathcal{S}_n$ is {\it block-wise simple} if there is $\sigma \in\mathcal{S}_k$ $(k \geq 4)$ which is a simple permutation (see Definition \ref{simple}), and there are $\alpha_1,\dots,\alpha_k$ which are block-wise simple permutations, such that $\pi=\sigma[\alpha_1,\dots,\alpha_k]$, the inflation of $\sigma$ by $\alpha_1,\dots,\alpha_k$ (see Definition \ref{inflation}). 
    
\end{enumerate}
\end{definition}

There are no block-wise simple permutations of orders $2$ and $3$.
For $n \in \{4,5,6\}$, a permutation is block-wise simple, if and only if it is simple. One of the first nontrivial examples of block-wise simple permutations is\\ $2413[3142,1,1,1]=4253716$.\\  

The following table, the calculations leading to most of its values are explained later in this paper, details the number of block-wise simple permutations and the number of simple permutations for $n$ less than or equal to $10$.\\ 
We denote by $W_n$ the set of all block-wise simple permutations of order $n$ and by $Simp_n$ the set of simple permutations of order $n$.  

\begin{center}
\begin{tiny}
\begin{tabular}{r||r|r|r|r|r|r|r|r|r|r|r|r}
   $n$ &  1 & 2 & 3 & 4 & 5 & 6 & 7 & 8 & 9 & 10 &11 &12 \\
  \hline\hline
 $|W_n|$  & 1 & 2  & 0  & 2  & 6  & 46 & 354  & 3034 & 29246 & 309174 & 3563562 & 44471970 \\
 \hline\hline
 $|Simp_n|$  & 1 & 2  & 0  & 2  & 6  & 46 & 338  & 2926 & 28146 & 298526  & 3454434 & 43286526  
\end{tabular}
\end{tiny}
\end{center}


In Section \ref{counting} we define the block-wise simple permutations in an alternative way and prove that the two definitions are equivalent (see Theorem \ref{equivalent}). 


Tenner \cite{T} defined the concept of an interval poset of a permutation (see Definition \ref{def interval poset}). This is an effective way of capturing all the intervals of a permutation and the set of inclusions between them in one glance. 
 An interval poset might correspond to more than one permutation. For instance, all simple permutations of a given order $n$ share the same interval poset. 
In this study, we count the number of interval posets that correspond to block-wise simple permutations. We prove that this number is exactly 
 
$$\frac{1}{n}\sum_{i=1}^{\lfloor\frac{n-1}{3}\rfloor}{n+i-1 \choose i}{n-2i-2 \choose i-1}$$

(for $n\geq 4$ ) (see Theorem \ref{num}).

We provide a bijection between the set of interval posets of block-wise permutations of order $n$ and the set of ways to place non-crossing diagonals in a convex ($n+1$)-gon to create no triangles or quadrilaterals.
 (see theorem \ref{bijection}). 
We also calculate the M\"obius function of the interval posets of block-wise simple permutations (see Theorem \ref{mob}).

Owing to the recursive definition of block-wise simple permutations, we obtain a generating function, the functional inverse of which is close to the generating function of simple permutations. This is proven in Section \ref{generating function}. \\

The two-sided Eulerian numbers, studied by Carlitz, Roselle, and Scoville, \cite{CRS}, count permutations according to their number of descents and inverse descents. 
Gessel (see \cite[Conjecture 10.2]{Branden})  conjectured that the bi-variate polynomial $\sum\limits_{\pi \in \mathcal{S}_n}s^{des(\pi)}t^{des(\pi^{-1})}$ can be written as a positive linear combination of elements of a certain 'gamma basis'. In \cite{Adin}, the authors of this paper, together with R. Adin, proposed a similar conjecture related to the set of simple permutations. Here, we show how to combinatorially prove 'gamma positivity' for block-wise simple permutations, based on the correctness of the conjecture on simple permutations. This is discussed in Section \ref{gamma positivity}.

\section{Background}\label{background}

\begin{definition}

Let $\pi=a_1 \cdots a_n \in \mathcal{S}_n$.   
An {\em interval} (or {\em block}) of $\pi$ is a non-empty contiguous
sequence of entries $a_i a_{i+1} \cdots a_{i+k}$ whose values also form a contiguous sequence of integers.
For $a<b$, $[a,b]$ denotes the interval of values that range from $a$ to $b$. 
Clearly, $[n]:=[1,n]$ is an interval as well as $\{i\}$ for each $i \in [n]$. These are called trivial intervals. The other intervals are called {\it proper}.
\end{definition}

\begin{example}
The permutation $\pi=314297856$ has $[5,9]=97856$ as a proper interval as well as the following proper intervals: $[1,4],[5,6],[7,8],[7,9]$, $[5,8]$. 
\end{example}

\begin{definition}
\label{simple}
A permutation $\pi \in \mathcal{S}_n$ is called {\em simple} if it does not have proper intervals. 
\end{definition}

\begin{example}\label{ex:simple up to 5}
The permutation $3517246$ is simple. 
\end{example}

\begin{definition}

Let $\pi \in \mathcal{S}_m$ and $\sigma \in \mathcal{S}_n$. 
The {\em direct sum} of $\pi$ and $\sigma$ is the permutation $\pi \oplus \sigma \in \mathcal{S}_{m+n}$ defined by
\[
(\pi \oplus \sigma)_i =
\begin{cases} 
\pi_i, & \text{if } i \leq m; \\
\sigma_{i-m}+m, & \text{if } i>m,
\end{cases}
\]
Their {\em skew sum} is the permutation $\pi \ominus \sigma \in \mathcal{S}_{m+n}$ defined by
\[
(\pi \ominus \sigma)_i = 
\begin{cases} 
\pi_{i}+n, & \text{if } i \leq m; \\
\sigma_{i-m},  & \text{if } i>m.
\end{cases}
\]
\end{definition}
\begin{example}
If $\pi=132$ and $\sigma=4231$ then $\pi \oplus \sigma=1327564$ and $\pi \ominus \sigma=5764231$.
\end{example}

The $\oplus$ and $\ominus$ operations are the simplest examples of inflation.

\begin{definition}
\label{inflation}
Let $n_1, \ldots, n_k$ be positive integers, with $n_1 + \ldots + n_k = n$.
The {\em inflation} of a permutation $\pi \in \mathcal{S}_k$ by the permutations $\alpha_i \in \mathcal{S}_{n_i}$ $(1 \leq i \leq k)$ is 
the permutation $\pi[\alpha_1, \ldots, \alpha_k] \in \mathcal{S}_n$ obtained by replacing for each $1 \leq i \leq k$ the $i$-th entry of $\pi$ with a block that is order-isomorphic to the permutation $\alpha_i$
on the numbers $\{s_i + 1, \ldots, s_i + n_i\}$ instead of $\{1, \ldots, n_i\}$, where $s_i = n_1 + \ldots + n_{i-1}$ $(1 \leq i \leq k)$. 
\end{definition}
It should be noted that $\pi \oplus \sigma = 12[\pi,\sigma]$ and $\pi \ominus \sigma = 21[\pi,\sigma]$. 

\begin{example}
The inflation of $2413$ by $213,21,132$ and $1$ is 
\[
2413[213,21,132,1]=546 \ 98 \ 132 \ 7.
\]
\end{example}

We present the following result from \cite{AA05}, which lays out the structure of permutations as inflation's of simple permutations.
\begin{thm}\label{Atkinson}
For each $n \geq  2$ and any $\pi \in \mathcal{S}_n$ there exists a unique simple permutation $\sigma \in \mathcal{S}_k$ such that 
$\pi = \sigma [ \alpha_1, \dots,\alpha_k] $. Moreover, if $k \geq 4$ then $\alpha_1, . . . , \alpha_k$ are uniquely determined.
\end{thm}

\section{Enumeration of block-wise simple permutations}\label{counting}
\subsection{An alternate definition of block-wise simple permutations} Recall Definition \ref{def of blockwise simple}.
We begin this section with another definition of block-wise simple permutations. 
\begin{definition}\label{def by 2 blocks}
A permutation $\pi \in \mathcal{S}_n$ is block-wise simple if and only if it has no interval of the form $p_1\oplus p_2$ or $p_1 \ominus p_2$. 
\end{definition}

\begin{thm} \label{equivalent} 
The two definitions of block-wise simple permutations are equivalent. 
\end{thm}
\begin{proof}
If $\pi\in \mathcal{S}_n$ is block-wise simple according to definition \ref{def of blockwise simple}, 
it is clear that $\pi$ contains no interval of the form $p_1\oplus p_2$ or $p_1 \ominus p_2$.

On the other hand, let $\pi\in \mathcal{S}_n, n>1$, be such that $\pi$ has no interval of the form $p_1\oplus p_2$ or $p_1 \ominus p_2$. Based on theorem \ref{Atkinson}, we can write $\pi=\sigma_1[\alpha_1^1,\dots,,\alpha_k^1]$ with $\sigma_1\in Simp_k$ and assume to the contrary that $\pi$ is not block-wise simple. Because $\pi$ has no intervals of the form $p_1\oplus p_2$ or $p_1 \ominus p_2$, we must have  $k\geq 4$ and there is some $1\leq i \leq k$ such that $\alpha_i^1$ is not block-wise simple. Assuming without loss of generality that $i=1$, we can write $\alpha_1^1=\sigma_2[\alpha_1^2,\dots \alpha_l^2]$ . 
As $\pi$ has no intervals of the form $p_1\oplus p_2$ or $p_1 \ominus p_2$, $\sigma_2$ is also simple of order $l\geq 4$. Without loss of generality, $\alpha_1^2$ is not block-wise simple. Proceeding in this manner, we obtain a sequence of non-block-wise simple permutations $(\alpha_1^j)$ which are inflation's of simple permutations, such that the sequence of the orders $(|\alpha_1^j|)$ is strictly decreasing. Hence, some $j$ exists, such that $|\alpha_1^j|=1$ whereas $\alpha_1^j$ is not block-wise simple. However, this is a contradiction. 
\end{proof}

\subsection{Counting block-wise simple permutations}

We now discuss the counting of block-wise simple permutations, starting with some values for low orders.

\begin{example}[$n=7$]\label{n=7}

Each simple permutation $\sigma$ of order $7$ induces a block-wise simple permutation of the form $\pi=\sigma[1,1,1,1,1,1,1]$ which is obviously simple. In this manner, we obtain $338$ block-wise simple permutations.  (see Fig. \ref{fig:n7Tree}, right).
Moreover, for each block-wise simple permutation $\tau$ of order $4$, and for each simple permutation $\sigma$ of order $4$, all permutations $$\pi\in \{\sigma[\tau,1,1,1],\sigma[1,\tau,1,1],\sigma[1,1,\tau,1],\sigma[1,1,1,\tau]\}$$ are block-wise simple. These results in $16$ new permutations, which sums up to $354$ block-wise simple permutations, together with the simple permutations of order $7$.   (see Fig. \ref{fig:n7Tree}, left).
\end{example}

 \begin{figure}
     \centering

 \begin{tikzpicture}
       \tikzstyle{every node} = [rectangle]
      
         \node (4) at (-4,1) {$4$};
            \node (41) at (-5.5,0) {$4$};
            \node (42) at (-4.5,0) {$1$};
            \node (43) at (-3.5,0) {$1$};
            \node (44) at (-2.5,0) {$1$};
           
          \node (exp4) at (-4,-1) {$|Simp_4| \cdot |Simp_4| \cdot 4 = 16$};

        \node (7) at (4,1) {$7$};
            \node (71) at (1,0) {$1$};
            \node (72) at (2,0) {$1$};
            \node (73) at (3,0) {$1$};
            \node (74) at (4,0) {$1$};
            \node (75) at (5,0) {$1$};
            \node (76) at (6,0) {$1$};
            \node (77) at (7,0) {$1$};
            
          \node (exp7) at (4,-1) {$|Simp_7|=338$};  
            
        \foreach \from/\to in {4/41,4/42,4/43,4/44,7/71,7/72,7/73,7/74,7/75,7/76,7/77}
            \draw[->] (\from) -- (\to);
    \end{tikzpicture}
 
\caption{ $n=7$.  }
     \label{fig:n7Tree}
 \end{figure}

\begin{example}[$n=8$]
For each simple permutation $\sigma$ of order $8$, we have the  block-wise simple permutation of the form $\pi=\sigma[1,1,1,1,1,1,1,1]$ . This contributes $2926$ block-wise simple permutations which are also simple.   
Moreover, for each block-wise simple permutation $\tau$ of order $5$ and for each simple permutation $\sigma$ of order $4$, all the permutations $$\pi\in \{\sigma[\tau,1,1,1],\sigma[1,\tau,1,1],\sigma[1,1,\tau,1],\sigma[1,1,1,\tau]\}$$ are block-wise simple. The contribution is $48$. 

Moreover, for each block-wise simple permutation $\tau$ of order $4$ and for each simple permutation $\sigma$ of order $5$, all the permutations:\\ $$\pi\in \{\sigma[\tau,1,1,1,1],\sigma[1,\tau,1,1,1],\sigma[1,1,\tau,1,1],\sigma[1,1,1,\tau,1],\sigma[1,1,1,1,\tau] \}$$ are block-wise simple. 
The total contribution is $60$. 
The total of block-wise permutations of order $8$ is thus $3034$ (see Fig. \ref{fig:n8Tree}).
 \end{example}

\begin{figure}
     \centering

 \begin{tikzpicture}
       \tikzstyle{every node} = [rectangle]
      
         \node (5) at (-4,1) {$5$};
            \node (51) at (-6,0) {$4$};
            \node (52) at (-5,0) {$1$};
            \node (53) at (-4,0) {$1$};
            \node (54) at (-3,0) {$1$};
            \node (551) at (-2,0) {$1$};
           
          \node (exp5) at (-4,-1) {}; 
          \node (exp54) at (-4,-1) {$|Simp_5| \cdot |Simp_4| \cdot 5 = 60$};
       
        \node (8) at (4,1) {$8$};
            \node (81) at (0.5,0) {$1$};
            \node (82) at (1.5,0) {$1$};
            \node (83) at (2.5,0) {$1$};
            \node (84) at (3.5,0) {$1$};
            \node (85) at (4.5,0) {$1$};
            \node (86) at (5.5,0) {$1$};
            \node (87) at (6.5,0) {$1$};
            \node (88) at (7.5,0) {$1$};
          \node (exp8) at (4,-1) {};
         \node (exp81) at (4,-1) {$|Simp_8| = 2926$};
       
         \node (4) at (0,-3) {$4$};
            \node (41) at (-1.5,-4) {$5$};
            \node (42) at (-0.5,-4) {$1$};
            \node (43) at (0.5,-4) {$1$};
            \node (44) at (1.5,-4) {$1$};
           
          \node (exp4) at (0,-5) {};
           \node (exp54) at (0,-5) {$|Simp_4| \cdot |Simp_5| \cdot 4 = 48$};
            
        \foreach \from/\to in {5/51,5/52,5/53,5/54,5/551,8/81,8/82,8/83,8/84,8/85,8/86,8/87,8/88,4/41,4/42,4/43,4/44}
            \draw[->] (\from) -- (\to);
\end{tikzpicture}
 
\caption{$ n=8. $}
     \label{fig:n8Tree}
\end{figure}

 From these two examples, we can determine the general pattern given by the following recursion: 
 \begin{obs}\label{formula}
 Let $w_n=|W_n|$ and $s_n=|Simp_n|$. Then, for $n\geq 4$ we have 
 \begin{equation}\label{recursion}
 w_n=\sum_{l=4}^n{s_l\sum_{\lambda=(\lambda_1,\dots,\lambda_l)\in Comp(n,l)}w_{\lambda_1}\cdots w_{\lambda_l}}
 \end{equation}
 where $Comp(n,l)$ is the set of compositions of $n$ in $l$ parts. \\
\end{obs}

Two remarks are now in order. 
\begin{rem}
Note that in Equation (\ref{recursion}), simple permutations are counted where $l=n$. 
\end{rem}

\begin{rem}
Note that in Example \ref{n=7}, we counted a limited number of compositions of $7$. This is reflected in the fact that, for each $l \in \{5,6\}$, a composition of $7$ of the form $\lambda=(\lambda_1,\dots,\lambda_l)$ contributes $0$ because each composition must contain some $\lambda_i\in\{2,3\}$ such that $w_{\lambda_1}=0$.  
\end{rem}

\section{Counting the number of interval posets}

Following Tenner \cite{T}, we define an {\it interval poset}
for each permutation as follows:

\begin{definition} \label{def interval poset}
The interval poset of a permutation $\pi \in S_n$ is the poset $P(\pi)$ whose elements are the non-empty intervals of $\pi$; the order is defined by set inclusion (see for example Figure \ref{fig:IntervalPoset5123647}). 
The minimal elements are the intervals of size $1$.  
\end{definition}

In \cite{T}, the interval poset is embedded in the plane so that each node's direct descendants are increasingly ordered according to the minimum of each interval from left to right. 
We note that in \cite{BCL} another embedding of the same poset was presented.

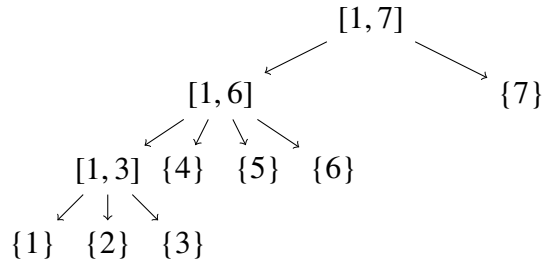
\begin{figure}[!ht]
     \centering

 \begin{tikzpicture}
       \tikzstyle{every node} = [rectangle]
      
         \node (17) at (0,0) {$[1,7]$};
            \node (16) at (-2,-1) {$[1,6]$};
            \node (7) at (2,-1) {$\{7\}$};
            \node (13) at (-3.5,-2) {$[1,3]$};
            \node (4) at (-2.5,-2) {$\{4\}$};
            \node (5) at (-1.5,-2) {$\{5\}$};
            \node (6) at (-0.5,-2) {$\{6\}$};
            \node (1) at (-4.5,-3) {$\{1\}$};
            \node (2) at (-3.5,-3) {$\{2\}$};
            \node (3) at (-2.5,-3) {$\{3\}$};
            
        \foreach \from/\to in {17/16,17/7,16/13,16/4,16/5,16/6,13/1,13/2,13/3}
            \draw[->] (\from) -- (\to);
    \end{tikzpicture}
 
\caption{Interval poset of the permutations: 5123647, 5321647, 4612357, 4632157, 7463215, 7461235, 7532164, 7512364 }
   \label{fig:IntervalPoset5123647}
 \end{figure}

If $\pi$ is a simple permutation, the interval poset of $\pi$ comprises the entire interval $[1,\dots,n]$ with minimal elements $\{1\},\dots,\{n\}$ as its only descendants. Hence, all simple permutations of a given order $n$ share the same interval poset (see for example Figure \ref{fig:IntervalPoset3142}).

\begin{figure}[!ht]
     \centering

 \begin{tikzpicture}
       \tikzstyle{every node} = [rectangle]
      
         \node (14) at (0,0) {$[1,4]$};
            \node (1) at (-1.5,-1) {$\{1\}$};
            \node (2) at (-0.5,-1) {$\{2\}$};
            \node (3) at (0.5,-1) {$\{3\}$};
            \node (4) at (1.5,-1) {$\{4\}$};

        \foreach \from/\to in {14/1,14/2,14/3,14/4}
            \draw[->] (\from) -- (\to);
    \end{tikzpicture}
 
\caption{Interval poset of permutations 3142 and 2413.  }
     \label{fig:IntervalPoset3142}
 \end{figure}
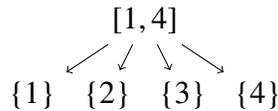

\begin{definition}
 Given a poset $P$, the permutations, whose interval poset is $P$ are called in \cite{T} the {\it generators} of $P$. 
 Formally, a set of such permutations is defined as
 $$I(P ) := \{w\in S_n : P(w) = P\}.$$
\end{definition}

Following \cite{T}, we define a {\it dual-claw poset} as a poset that has a unique maximal element and $k \geq 4$ minimal
elements, which are all covered by the maximal element. For our convenience, we omit the adjective 'dual' and call it here just {\it claw poset }.

A claw poset with $k$ minimal elements is an interval poset of a permutation $\sigma$ if and only if $\sigma$ has no proper intervals; that is, $\sigma$ is simple. Thus, its generators are the simple permutations of order $k$.
\begin{rem}\label{claw of claws}
A {\it tree poset} is a poset whose Hasse diagram is a tree.
In \cite{T} (Theorem 6.1), the author claimed that $P(\sigma)$ is a tree interval poset if and only if $\sigma$ contains no interval of the form $p_1\oplus p_2\oplus p_3$ or $p_1\ominus p_2 \ominus p_3$. 
Based on Theorem \ref{equivalent}, it is evident that the interval poset of a block-wise simple permutation is a tree. 
In fact, by Definition \ref{def of blockwise simple}, the interval poset of a block-wise simple permutation is a {\it claw of claws}. This means that the root is a claw and every node is either a claw or a leaf (see Figure. \ref{fig:DualClaws} for an example).
\end{rem}

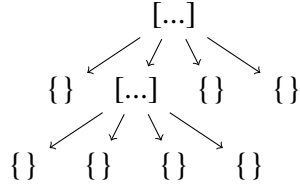
\begin{figure}[!ht]
     \centering

 \begin{tikzpicture}
       \tikzstyle{every node} = [rectangle]
      
         \node (1) at (0,0) {$[...]$};
            \node (11) at (-1.5,-1) {$\{\}$};
            \node (12) at (-0.5,-1) {$[...]$};
            \node (13) at (0.5,-1) {$\{\}$};
            \node (14) at (1.5,-1) {$\{\}$};
            \node (121) at (-2,-2) {$\{\}$};
            \node (122) at (-1,-2) {$\{\}$};
            \node (123) at (0,-2) {$\{\}$};
            \node (124) at (1,-2) {$\{\}$};
            
        \foreach \from/\to in {1/11,1/12,1/13,1/14,12/121,12/122,12/123,12/124}
            \draw[->] (\from) -- (\to);
    \end{tikzpicture}
 
\caption{Claw of claws }
     \label{fig:DualClaws}
 \end{figure}

  The first problem we tackle in this section is the enumeration of interval posets that represent block-wise simple permutations of order $n$. As described above, this problem reduces to counting claws of claws.
 
 Let $\mathcal{C}$ be the combinatorial class of all claws of claws, so that for each $n$, $\mathcal{C}_n$ is the set of all claws of claws having $n$ leaves. 
 We apply the approach of symbolic combinatorics (see \cite{FS}) to count claw of claws.

The class $\mathcal{C}$ is described by the following symbolic combinatorial equation: $$\mathcal{C}=\{\bullet\} \biguplus Seq_{\geq 4}(\mathcal{C})$$  where $\{\bullet\}$ represents a leaf and $Seq_{\geq 4}(\mathcal{C})$ represents a sequence of elements of $\mathcal{C}$ with at least four components. 
The generating function is 
    $$C(z)=z+\frac{C(z)^4}{1-C(z)}.$$ 
    
Denoting $u=C(z)$, we have $$z=u-\frac{u^4}{1-u}=u(1-\frac{u^3}{1-u}).$$
Subsequently, $u=z\phi(u)$ is obtained, where $\phi(u)=\frac{1}{1-\frac{u^3}{1-u}}$.

According to Lagrange inversion formula (see for example \cite{EC2}[Theorem 5.4.2]), we have
$$C(z)=\sum_{n\geq 1} c_n z^n $$ where $c_n=\frac{1}{n}[u^{n-1}]\phi(u)^n$. \\

To extract $[u^{n-1}]\phi(u)^n$ we compute:

$$\phi(u)^n=\left (\frac{1}{1-\frac{u^3}{1-u}}\right)^n=$$

$$\sum_{i\geq 0}{n+i-1 \choose i}\frac{u^{3i}}{(1-u)^i}=1+\sum_{i>0}{n+i-1 \choose i}\frac{u^{3i}}{(1-u)^i}=$$
$$1+\sum_{i> 0}{n+i-1 \choose i}{u^{3i}}\sum_{j \geq 0}{i+j-1 \choose j}u^j=1+\sum_{i>0}\sum_{j \geq 0}{n+i-1 \choose i}{i+j-1 \choose j}u^{j+3i}.$$

By extracting the coefficient of $u^{n-1}$, we obtain $j=n-1-3i$

$$[u^{n-1}]\phi(u)^n=\sum_{i>0}{n+i-1 \choose i}{i+n-2-3i \choose n-3i-1}=$$
$$\sum_{i>0}{n+i-1 \choose i}{n-2i-2 \choose n-3i-1}=$$
$$=\sum_{i=1}^{\lfloor\frac{n-1}{3}\rfloor}{n+i-1 \choose i}{n-2i-2 \choose i-1}$$
Because we require $n-2i-2\geq i-1$, we obtain $i\leq \frac{n-1}{3}$. \\

From the above calculations, we obtain the following theorem:
\begin{thm}\label{num}
The number of interval posets corresponding to block-wise simple permutations of order $n\geq 4$ is 
$$\frac{1}{n}\sum_{i=1}^{\lfloor\frac{n-1}{3}\rfloor}{n+i-1 \choose i}{n-2i-2 \choose i-1}.$$

\end{thm}

The first few values of the sequence of these numbers  are $1,1,1,5,10,16,45,109,222,540$. This is sequence A054514 from OEIS \cite{Sl} which also counts the number of ways to place non-crossing diagonals in a convex $(n+4)$-gon such that there are no triangles or quadrilaterals.\\
Here, we provide a combinatorial proof presenting a bijection from the set of interval posets corresponding to block-wise simple permutations to the above-mentioned set.

We identify a polygon with its set of vertices and denote a diagonal from vertex $i$ to vertex $j$ by $(i,j)$.

\begin{figure}[!ht]
\label{interval}
    \centering
    \includegraphics[scale=0.45]{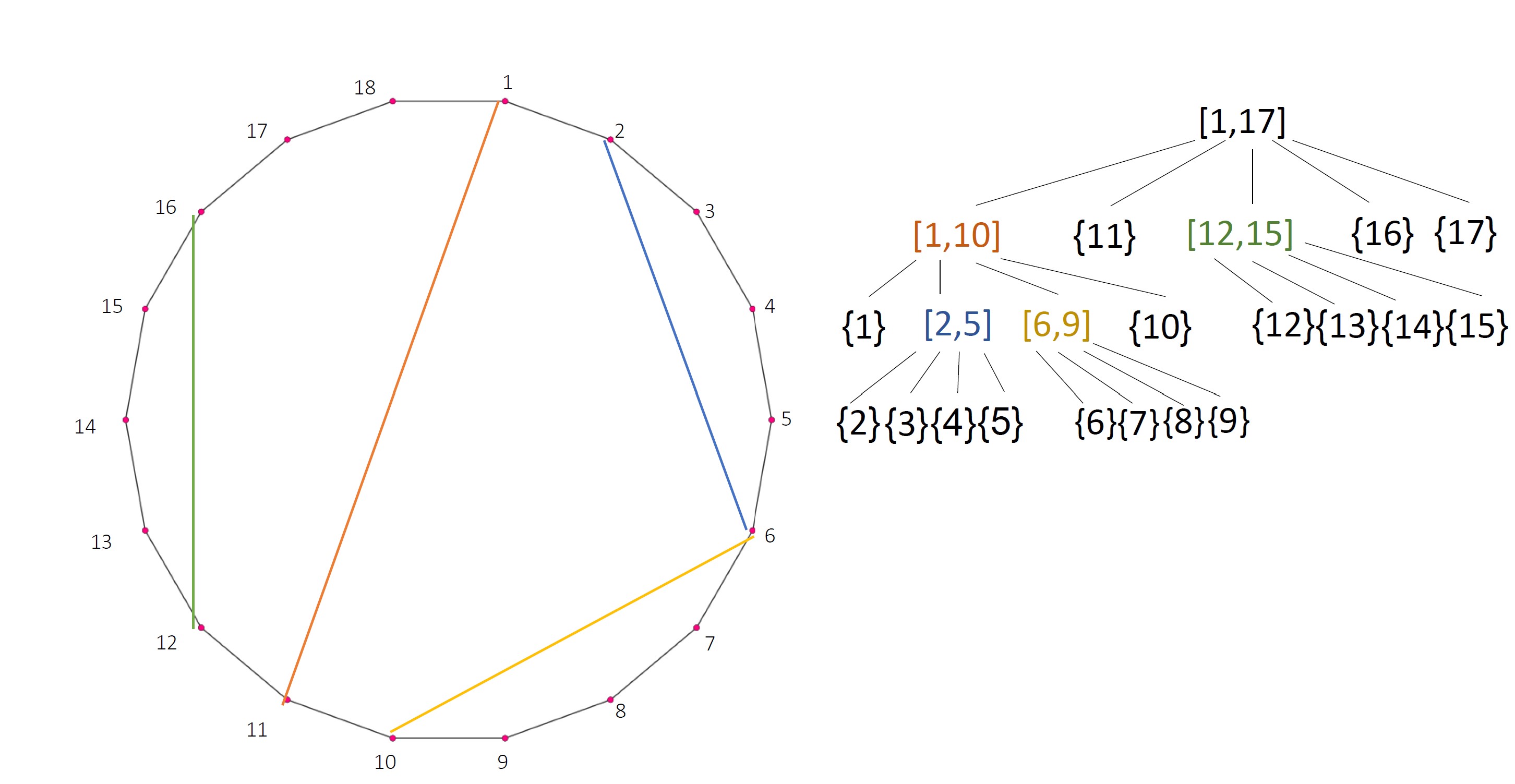}
    \caption{Right: the interval poset P. Left: the polygon $\Phi(P)$}
    \label{polygon to poset}
\end{figure}

We let $P$ be the interval poset corresponding to a block-wise simple permutation $\pi$ of order $n$. We set $\Phi(P)$ as the convex ($n+1$)-gon whose set of diagonals is $$\{(a,b+1)|\,[a,b] \text{ is an internal node of } P\}.$$ (see Figure \ref{polygon to poset} for an example).

Note that by Remark \ref{claw of claws}:
\begin{enumerate}
    \item For every two nodes $[a,b],[c,d]$ of $P$, either $[a,b]\subset [c,d]$ or $[c,d]\subset [a,b]$ or $[a,b]\cap [c,d]= \emptyset$.
    \item Each node $[a,b]$ of $P$ is an interval of size at least $4$ so $b-a \geq 3$.
    \item
    Each node $[a,b]$ of $P$ covers at least $4$ nodes, which are pairwise disjoint.
\end{enumerate}

According to (1), the polygon $\Phi(P)$ has no crossing diagonals. 
By (3), each direct descendant of an interval $[a,b]$ which is itself an interval, say $[c,d]$, creates a sub-polygon $\{c,\dots,d,d+1\}$ of size $d-c+2>4$.
There is an additional sub-polygon created by the direct descendants of $[a,b]$ that are not intervals, together with the first elements of each interval and the vertex $b+1$. By (2), this sub-polygon also has at least $5$ vertices. 
We conclude that the diagonals of the sub-polygon associated with each interval $[a,b]$ create no triangles or quadrilaterals.

\begin{example}
Consider $[a,b]=[1,17]$ in Figure \ref{polygon to poset}. The corresponding polygon is $E=\{1,\dots,18\}$. The descendants of $[1,17]$ are $[1,10]$, $\{11\}$, $[12,15]$, $\{16\}$ and $\{17\}$.

The interval $[1,10]$  creates the sub-polygon $\{1,\dots,11\}$ of order $11$. The interval $[12,15]$ creates the sub-polygon $\{12,\dots ,16\}$. Both are of size greater than $4$ by the condition on the size of intervals. The additional sub-polygon $\{1,11,12,16,17,18\}$ is of size greater than $4$ because each of the descendants of $[1,17]$ contributes at least one vertex (its first vertex if it is an interval, or itself if it is a singleton) and including the vertex $18$ makes it greater than $4$.
\end{example}
From the discussion above  one can prove the following:



\begin{thm}\label{bijection}
The number of interval posets that represent a block-wise simple permutation of order $n$ is equal to the number of ways to place non-crossing diagonals in a convex ($n+1$)-gon such that no triangles or quadrilaterals are present. 
\end{thm}

\begin{rem}
In \cite{BCL}, the authors calculated the generating function of the number of tree interval posets using generating functionology and mentioned that this is equal to the number of ways to place non-crossing diagonals in a convex $(n+2)$-gon such that no quadrilaterals are created (sequence A054515 from OEIS \cite{Sl2}).
The function $\Phi$ defined above can be slightly amended so that it would provide a combinatorial proof of this equivalence. 
\end{rem}

We end this section with a simple consequence of Theorem 2.8 of \cite{BCL}, concerning M\"obius function of the interval poset of a block-wise simple permutation.
\begin{prop}
\label{mob}
Let $\sigma=\pi[\alpha_1,\dots,\alpha_k]$ be a block-wise simple permutation and let $P(\sigma)$ be its closed interval poset. 
For each interval $I$ of $P(\sigma)$ we have

$$\mu(I,[1,\dots,n])=\left\{ \begin{array}{lll}
1 & &I=[1,\dots,n]
  \\
-1 & & I \text{ is a coatom} \\
k-1 & & I=\emptyset \\
0   & &  \text{Otherwise}
\end{array}\right. $$
\end{prop}

\section{A generating function for block-wise simple permutations}\label{generating function}
Recall that a permutation $\pi$ is called 'simple' if it has no blocks other than the singletons and the entire permutation.

In most sources in the literature, permutations of orders $1$ and $2$ are considered simple. However, there are cases in which simple permutation is defined as having an order of at least $4$ (see for example  \cite{AAK}). In this section, we adopt the latter approach.

Let $S(x)=\sum\limits_{n\geq 4}s_n x^n$ be the generating function of the simple permutations and $W(x)=\sum\limits_{n\geq 1}w_n x^n$ be the generating function of block-wise simple permutations. The following recursion is a direct consequence of Definition \ref{def of blockwise simple}. 
\begin{equation}\label{gf 1}
W(x)=x+(S\circ W)(x).
\end{equation}

Let $F(x)=\sum\limits_{k\geq 1}n!x^n$.
By denoting $t=W(x)$ and thus $x=W^{-1}(t)$ in Equation (\ref{gf 1}), we obtain for each $t$
\begin{equation}\label{st}
S(t)=t-W^{-1}(t).
\end{equation}
From Equation (2) of \cite{AAK}, we obtain  
\begin{equation}\label{klazar}
S(t)=t-\frac{2t^2}{1+t}-F^{-1}(t)
\end{equation}

and by substituting Equation (\ref{klazar}) in Equation 
(\ref{st}) we obtain 
\begin{equation}\label{final formula}
W^{-1}(t)=F^{-1}(t)+\frac{2t^2}{1+t}.
\end{equation}
Denoting the coefficient of $t^n$ in $F^{-1}(t)$ by $Com_n$ (in reference to Comtet who was the first to consider this sequence OEIS(A059372)), and letting $W^{-1}(t)=\sum\limits_{n=1}^{\infty}a_nt^n$, we obtain $a_1=1,a_2=0,a_3=0$ and for $n\geq 4$:
$$a_n=Com_n+2\cdot (-1)^n.$$ 
From (\ref{st}), we obtain for $n\geq 4$: $$a_n=-s_n.$$

By applying the Lagrange inversion formula \cite{EC2} [Theorem 5.42], we obtain an expression for the $n-th$ coefficient of $W(t)$:
$$[t^n]W(t)=\frac{1}{n}[t^{n-1}]\left(\frac{t}{W^{-1}(t)}\right)^n$$
Thus, $$[t^n]W(t)=\frac{1}{n}[t^{n-1}]\left(\frac{t}{1+\sum\limits_{k=2}^{\infty}(Com_k+2\cdot (-1)^k)t^k}\right)^n$$

From (\ref{st}), we can conclude that the sequence of block-wise simple permutations is not polynomial recursive.

\section{gamma positivity}\label{gamma positivity}

Eulerian numbers enumerate permutations according to their descent numbers. The descent number of a permutation is defined as follows:
\begin{definition}
Let $\pi \in \mathcal{S}_n$. 
\[
\des(\pi) = |\{i \in [n-1] \mid \pi(i)>\pi(i+1)\}|.
\]
The descent number of the inverse of a permutation $\pi$ is defined as $$\ides(\pi)=\des(\pi^{-1}).$$ 

\end{definition}
For example, if $\pi= 2 4 6 1 3 5$ then the only descent of $\pi$ is $3$; thus, we obtain $\des(\pi)=1$, whereas $\ides(\pi)=3$. 
Many studies have been made on the distribution of descents on certain subsets of $\mathcal{S}_n$ and on some other permutation groups.

The {\em two-sided Eulerian numbers}, studied by Carlitz, Roselle, and Scoville \cite{CRS} constitute a natural generalization. These numbers count permutations according to their number of descents as well as the number of descents of the inverse permutation.

Gessel conjectured that the corresponding generating function $A_n(s,t)=\sum_{\pi \in S_n}s^{\des(\pi)}t^{ides(\pi)}$ has a nice symmetry property, namely, gamma-positivity. His conjecture was proven in \cite{Lin}. 
To present the idea of gamma-positivity in a formal way, we present some definitions. 

A polynomial $f(q)$ is {\em palindromic} if its coefficients are the same when read from left to right and from right to left. 
If $f(q)=a_rq^r+a_{r+1}q^{r+1} +\cdots +a_sq^s$ with $a_r, a_s \ne 0$ and $r \le s$, then $a_{r+i}=a_{s-i}$ $(\forall i)$; equivalently,
$f(q) = q^{r+s} f(1/q)$.

Following Zeilberger \cite{Z}, we define the {\it darga} of a palindromic polynomial $f(q)$ as $r+s$; the zero polynomial is considered palindromic for each non-negative darga. 
The set of palindromic polynomials of darga $n-1$ is a vector space of dimension $\lfloor (n+1)/2 \rfloor$, with a {\em gamma basis}
\[
\{q^j (1+q)^{n-1-2j} \mid 0 \le j \leq \lfloor (n-1)/2 \rfloor \}.
\]
The (one-sided) {\em Eulerian polynomial} 
\[
A_n(q) = \sum_{\pi \in S_n} q^{\des(\pi)}
\]
is palindromic of darga $n-1$ and thus there are real numbers $\gamma_{n,j}$ such that 
\[
A_n(q) = \sum_{0 \leq j \leq \lfloor (n-1)/2 \rfloor} \gamma_{n,j} q^j(1+q)^{n-1-2j}.
\]
See \cite[pp.\ 72 - 79]{Pet Book} for details.

For every set $\mathcal{C}$ of permutations, we define the {\em two-sided Eulerian polynomial} of $\mathcal{C}$ as
\[
A_\mathcal{C}(s,t) = \sum_{\pi \in \mathcal{C}} s^{\des(\pi)} t^{\ides(\pi)}.
\]

Let $A_{Simp_n}(s,t)$ and $A_{W_n}(s,t)$ be the {\em two-sided Eulerian polynomials for simple permutations and block-wise simple permutations of order $n$}, respectively.

%
 For each positive integer $n$, the set $Simp_n$ of simple permutations of length $n$ is invariant under taking inverses and reverses. Consequently,  the set $W_n$ of blockwise simple permutations of length $n$ is also invariant under taking inverses and reverses. 
 
 
Thus, it can be easily proven that the bi-variate polynomial $A_{W_n}(s,t)$  satisfies 

\begin{equation}\label{eq.inv}
A_{W_n}(s,t) = A_{W_n}(t,s)
\end{equation}
as well as
\begin{equation}\label{eq.pal}
A_{W_n}(s,t) = (st)^{n-1} A_{W_n}(1/s, 1/t).
\end{equation}

In fact, (\ref{eq.pal}) follows from the bijection from $W_n$ onto itself, transforming a permutation to its reverse, whereas (\ref{eq.inv}) follows from the bijection taking each permutation to its inverse.

A bivariate polynomial satisfying Equations  (\ref{eq.inv}) and (\ref{eq.pal}) is called (bivariate) {\em palindromic of darga} $n-1$.

It can be proven (see \cite[pp.\ 72-79]{Pet Book}) that the set of bivariate palindromic polynomials of darga $n-1$ is a vector space of dimension $\lfloor (n+1)/2 \rfloor \cdot \lfloor (n+2)/2 \rfloor$ with {\em bivariate gamma basis}:
\[
\{(st)^i(s+t)^j(1+st)^{n-1-j-2i} \mid i,j \ge 0,\, 2i+j \le n-1 \}.
\]
\begin{definition}
A bivariate palindromic polynomial is called {\em gamma-positive} if all coefficients in its expression in terms of the bivariate gamma basis are non-negative integers. 
\end{definition}
\begin{example}
Recall that 
\begin{enumerate}
    \item $Simp_4=\{2413,3142\}$. The two-sided Eulerian polynomial $$A_{Simp_4}(s,t)=s^2t+st^2=(st)(s+t)$$ is gamma-positive.
    \item The polynomial $$A_{Simp_6}(s,t) = st(s+t)^2(1+st)+5(st)^2(1+st)+14(st)^2(s+t)$$ is also gamma-positive.
\end{enumerate}

\end{example}
In our previous paper \cite{Adin} we presented the following conjecture and used it for reducing the problem of gamma-positivity of the entire group $\mathcal{S}_n$ to that of $Simp_n$. 
\begin{con}
\label{conj:simple}
For each positive $n\in \mathbb{N}$, the two-sided Eulerian polynomial,
\[
A_{Simp_{n}}(s,t) = \sum_{\sigma \in Simp_n} s^{\des(\sigma)} t^{\ides(\sigma)}
\]
is gamma-positive.
\end{con}
Assuming that Conjecture \ref{conj:simple} holds, we can prove the following theorem:
\begin{thm}
\label{t:gamma_positivity_Wn} 

For each $n \geq 1$ there exist non-negative integers $\gamma_{n,i,j}$ $(i,j \ge 0,\, 2i+j \leq n-1)$ such that
\[
A_{W_n}(s,t)=\sum\limits_{i,j}{\gamma_{n,i,j}(st)^i(s+t)^j(1+st)^{n-1-j-2i}}.
\]
\end{thm}

To prove Theorem \ref{t:gamma_positivity_Wn}, we require some observations.

\begin{obs}\label{respect des}

Inflation is additive on both $\des$ and $\ides$. Explicitly:

Let $\sigma=\pi[\tau_1,\dots,\tau_k]$, then
\[
\des(\sigma)=\des(\pi)+\sum_{i=1}^n{\des(\tau_i)}
\]
and
\[
\ides(\sigma)=\ides(\pi)+\sum_{i=1}^n{\ides(\tau_i)}
\]
thus, 
\[
s^{\des(\sigma)}t^{\ides(\sigma)}=s^{\des(\pi)}t^{\ides(\pi)}\prod_{i=1}^n s^{\des(\tau_i)}t^{\ides(\tau_i)}.
\]
\end{obs}

\begin{defn}
Let $\mathcal{A}\subset S_k$ and $\mathcal{B}_1,\ldots,\mathcal{B}_k$ be sets of permutations. 
Let us define 
\[
\mathcal{A}[\mathcal{B}_1,\ldots,\mathcal{B}_k] 
= \{\alpha[\beta_1,\ldots,\beta_k] \mid \alpha \in \mathcal{A},\, \beta_i \in \mathcal{B}_i \}.
\]

\end{defn}

If $\mathcal{B}_i \subset S_{n_i}$ ($i \in \{1,\dots,k\}$), then $\mathcal{A}[\mathcal{B}_1,\ldots,\mathcal{B}_k] \subset S_n$ where $n=\sum n_i$. 
\begin{prop}
The two-sided Eulerian polynomial of $\mathcal{A}[\mathcal{B}_1,\ldots,\mathcal{B}_k]$ is the product of the two-sided Eulerian polynomials of $\mathcal{A}$ and $\mathcal{B}_1,\ldots,\mathcal{B}_k$.
\end{prop}
\begin{proof}
From Observation \ref{respect des}, the following can be obtained
\[
\begin{aligned}
\sum_{\sigma=\pi[\tau_1,\dots,\tau_k] \in \mathcal{A}[\mathcal{B}_1,\ldots,\mathcal{B}_k]}s^{\des(\sigma)}t^{\ides(\sigma)}&=\sum_{\pi \in \mathcal{A},  \tau_1 \in \mathcal{B}_1,\dots,\tau_k \in \mathcal{B}_k} \bigg(s^{\des(\pi)}t^{\ides(\pi)}\prod_{i=1}^n s^{\des(\tau_i)}t^{\ides(\tau_i)}\bigg) \\
&=\sum_{\pi \in \mathcal{A}}s^{\des(\pi)}t^{\ides(\pi)} \prod_{i=1}^n \bigg(\sum_{\tau_i \in \mathcal{B}_i}s^{\des(\tau_i)}t^{\ides(\tau_i)} \bigg)
\end{aligned}
\]
as required.
\end{proof}

\begin{obs}
\label{obs:gamma_positive_of_inflation}
If the two-sided Eulerian polynomials of $\mathcal{A}$ and $\mathcal{B}_1,\ldots,\mathcal{B}_k$ are gamma positive, then the two-sided Eulerian polynomial of $\mathcal{A}[\mathcal{B}_1,\ldots,\mathcal{B}_k]$ is also gamma positive.
\end{obs}
We proceed to the proof of Theorem \ref{t:gamma_positivity_Wn}.
\begin{proof}[Proof of Theorem \ref{t:gamma_positivity_Wn}]
Use induction on $n$.
For $n=1$ and $n=4$ we have $A_{W_1}(s,t)=1$ and $A_{W_4}(s,t)=(st)(s+t)$ which are gamma positive.

Let $n>4$ and assume that the polynomial $A_{W_k}(s,t)$ is gamma-positive for every $k<n$. Note that $W_n=Simp_n \cup U_n$ where
$U_n=\bigcup\limits_{4 \leq k \leq n-3 \atop n_1+\cdots n_k=n}Simp_k[W_{n_1}\ldots,W_{n_k}]$. Therefore, the two-sided Eulerian polynomial of $W_n$ is the sum of the two-sided Eulerian polynomial of $Simp_n$ and the two-sided Eulerian polynomials of $Simp_k[W_{n_1}\ldots,W_{n_k}]$. By Conjecture \ref{conj:simple}, the induction hypothesis, and Observation \ref{obs:gamma_positive_of_inflation}, these polynomials are gamma-positive and we are done.

\end{proof}

\section{Asymptotics- an open question and a conjecture}

One might use Equation (\ref{recursion}) to obtain an upper bound for the proportion of the number of blockwise simple permutations to the magnitude of the entire set of permutations. Thus, we write 
for each $n\in \mathbb{N}$, $A_n=W_n-Simp_n$ and we are interested in the asymptotic behavior of $R_n=\frac{|A_n|}{|S_n|}$.
Experimental checks show that this ratio tends to zero when $n$ tends to infinity as can be seen in the following table. Actually, we have corroborating data up to $n=23$.

\begin{center}
\begin{tiny}
\begin{tabular}{r||r|r|r|r|r|r|r|r|r|r|r|r}
   $n$ &  1 & 2 & 3 & 4 & 5 & 6 & 7 & 8 & 9 & 10 &11 &12 \\
  \hline\hline
$R_n$  & 0 & 0  & 0 & 0 & 0 & 0 & 0.0031746 & 0.00267857 & 0.00303131 & 0.0029343 & 0.00273389 & 0.00247482
\end{tabular}
\end{tiny}
\end{center}

so we have the following conjecture:

\begin{con}
The proportion $R_n$ tends to $0$ when $n$ tends to infinity.  
\end{con}

\section*{Acknowledgments}
The authors want to thank Ruth Hoffman and her team for helping us in the use of their software that contributed a lot to the expermiental side of this work.

%

\end{document}